\documentclass[a4paper, 11pt, reqno]{amsart}
\usepackage[usenames,dvipsnames]{color}
\usepackage{amssymb, amsmath, latexsym, graphics, graphicx}
\newtheorem{theorem}{Theorem}[section]
\newtheorem{lemma}[theorem]{Lemma}
\newtheorem{corollary}[theorem]{Corollary}
\newtheorem{example}[theorem]{Example}
\newtheorem{proposition}[theorem]{Proposition}

\makeatother

\newcommand{\ul}{\underline}

\def\D{{\mathcal D}}

\newcommand\trop{\mathbb{T}}

\newcommand\ft{\mathbb{FT}}
\newcommand\utnt{UT_n(\trop)}
\newcommand\uttt{UT_2(\trop)}

\begin{document}
\title[Identities in Tropical Matrix Semigroups]{Identities in Upper Triangular Tropical Matrix Semigroups and the Bicyclic Monoid}

\maketitle

\begin{center}

LAURE DAVIAUD\footnote{
Faculty of Mathematics, Informatics and Mechanics, University of Warsaw, Poland. Email \texttt{ldaviaud@mimuw.edu.pl}. Work supported by the LIPA project, funded by the European Research Council (ERC) under the European Union's Horizon 2020 research and innovation programme (grant agreement
No 683080).},
MARIANNE JOHNSON\footnote{School of Mathematics, University of Manchester,
Manchester M13 9PL, UK. Email \texttt{Marianne.Johnson@maths.manchester.ac.uk}.}
and
MARK KAMBITES\footnote{School of Mathematics, University of Manchester,
Manchester M13 9PL, UK. Email \texttt{Mark.Kambites@manchester.ac.uk}.}

\date{\today}
\thanks{}
\end{center}

\begin{abstract} We establish necessary and sufficient conditions for a semigroup identity to hold in the monoid of $n \times n$ upper triangular tropical matrices, in terms of equivalence of certain tropical polynomials. This leads to an algorithm for checking whether such an identity holds, in time polynomial in the length of the identity and size of the alphabet. It also allows us to answer a question of Izhakian and Margolis, by showing that the identities which hold in the monoid of $2 \times 2$ upper triangular tropical matrices are exactly the same as those which hold in the bicyclic monoid. Our results extend to a broader class of ``chain structured tropical matrix semigroups''; we exhibit a faithful representation of the free monogenic inverse semigroup within such a semigroup, which leads also to a representation by $3 \times 3$ upper triangular tropical matrices.
\end{abstract}

\vspace{1ex}
\noindent \textit{This amended version of the author accepted manuscript contains a corrected proof of Proposition 7.1. The error in the original published proof is described below the new proof. The new proof establishes the proposition exactly as originally published, all other results are unaffected and the manuscript is otherwise unamended. We are grateful to Duarte Ribeiro (private communication) for alerting us to the error.}

\section{Introduction}

Over the last few years there has been considerable interest in the structure of the monoid $M_n(\trop)$ of tropical (max-plus) matrices under multiplication, with lines of investigation including characterisations of Green's relations \cite{G98,HK12,JK11,JK13}, structural properties of subsemigroups and maximal subgroups \cite{dAP03, G96, IJK17}, as well as connections between the algebraic properties of elements (or indeed subsemigroups) and their actions upon tropically convex sets \cite{IJK16,M10,SLS13}.

A natural question, especially in view of d'Alessandro and Pasku's proof that finitely generated subsemigroups of $M_n(\trop)$ have polynomial growth \cite{dAP03}, is whether $M_n(\trop)$ satisfies a nontrivial semigroup identity. Izhakian and Margolis \cite{IM10} were the first to consider this question, in the $2 \times 2$ case; they showed that $M_2(\trop)$ does indeed satisfy a nontrivial identity. A key step of their proof was establishing an identity for the upper triangular submonoid $UT_2(\trop)$, which is also of interest in its own right. Specifically, they showed that $UT_2(\trop)$ satisfies the celebrated identity $AB^2A^2BAB^2A=AB^2ABA^2B^2A$, which was shown by Adjan \cite{A66} to hold in the bicyclic monoid $\mathcal{B}$. Since $\mathcal{B}$ embeds in $UT_2(\trop)$, every identity satisfied in the latter must also hold in the former. In view of this and their results, Izhakian and Margolis posed the natural question of whether the converse holds, that is, whether $UT_2(\trop)$ and $\mathcal{B}$ satisfy exactly the same identities, or equivalently (by Birkhoff's HSP theorem \cite{B35}), whether they generate the same variety.

In Section~\ref{sec_2x2identities} we develop an exact characterisation of identities which hold in $UT_2(\trop)$, by associating $k$ tropical polynomial equations in $k$ variables to each identity on $k$ letters. This gives a simple and algorithmically efficient (as we shall see in Section \ref{sec_complexity}) method to check whether any given identity holds. In Section~\ref{sec_bicyclic}, by considering the embedding of $\mathcal{B}$ into $UT_2(\trop)$, we use this to show that an identity which fails to hold in $UT_2(\trop)$ must also fail in $\mathcal{B}$, thus answering the above-mentioned question of Izhakian and Margolis \cite{IM10} and establishing that the bicyclic monoid $\mathcal{B}$ and the upper triangular tropical matrix monoid $UT_2(\trop)$ satisfy exactly the same semigroup identities (Theorem \ref{main_theorem}).

It follows by Birkhoff's HSP Theorem, of course, that $\mathcal{B}$ and $UT_2(\trop)$ generate the same variety, and hence share all properties of monoids which are visible in the variety generated. For example, by a theorem of Shneerson \cite{Shn89}, $\mathcal{B}$ has infinite axiomatic rank, so $UT_2(\trop)$ must also have.
In particular, since infinite axiomatic rank implies the non-existence of a finite basis, from our theorem plus \cite{Shn89} we can deduce the recent result of Chen, Hu, Luo and Sapir \cite{CHLS16} that the variety generated by $UT_2(\trop)$ shares with that generated by $\mathcal{B}$ the property of not admitting a finite basis of identities.

There is an interesting relationship between these results and work of Pastijn \cite{P06}, which studies identities in the bicyclic monoid by reduction to checking linear inequalities or, equivalently, properties
of certain polyhedral complexes. Once we have established that identities in $\mathcal{B}$ are equivalent to identities in $\uttt$, it becomes evident that our methods (for identities in $\uttt$) are related to Pastijn's (for identities in $\mathcal{B}$). However, since our methods are used to establish the correlation in the first place, we cannot hope to deduce our results from Pastijn's work.

In higher dimensions, understanding identities seems to be hard. Izhakian and Margolis have conjectured that $M_n(\trop)$ satisfies an identity for every $n$; evidence supporting this includes their own results on $M_2(\trop)$, and the proof by d'Alessandro and Pasku that finitely generated subsemigroups of $M_n(\trop)$ all have polynomial growth. More recently Shitov \cite{Shi14} has produced an identity for $M_3(\trop)$, over a $2$-letter alphabet and with 1,795,308 letters on each side. Identities satisfied by semigroups of tropical matrices satisfying non-singularity conditions were considered in \cite{I16, Shi14}. In \cite{I14} Izhakian claimed a family of identities for $UT_n(\trop)$ for each $n$; in fact the proof contains a technical error and at least some of the claimed identities do not hold, but they can be corrected \cite{I14Erratum,TaylorThesis} to give examples of identities which do hold in these semigroups. Further examples of identities for the upper triangular case have been constructed by Okninski \cite{O15} and by Cain \textit{et al} \cite{Cain17}; the latter work provides an additional motivation for studying upper triangular tropical matrices by exhibiting an embedding of the \textit{plactic monoid} of rank $3$  into the direct product $UT_3(\trop) \times UT_3(\trop)$.

In Section~\ref{sec_chain} we extend the ideas of Section~\ref{sec_2x2identities} to characterise identities which hold in $UT_n(\trop)$, and indeed in a more general class of \textit{chain structured tropical matrix semigroups} (to be defined below), in terms of equivalence of certain tropical polynomial functions. Section~\ref{sec_divisors} establishes some structural results about chain structured tropical matrix semigroups.
Section~\ref{sec_monogenicinverse} considers the free monogenic inverse semigroup $\mathcal{I}$; we exhibit a natural embedding of $\mathcal{I}$ into a chain structured tropical matrix semigroup which in turn embeds in $UT_3(\trop)$. By combining with the results of Section~\ref{sec_chain} we recover an alternative proof that $\mathcal{I}$ satisfies the same identities as $UT_2(\trop)$ (which via Theorem \ref{main_theorem}, is equivalent to the known fact that it satisfies the same identities as the bicyclic monoid).

Finally, Section~\ref{sec_complexity} applies the results of Sections~\ref{sec_2x2identities} and \ref{sec_chain} to the algorithmic problem of deciding whether a given identity holds in $UT_n(\trop)$. These results allow us to reduce the problem to (real) linear programming, and hence to show that, for a fixed $n$, it can be solved in time polynomial in the size of the identity and the alphabet over which it is defined. In particular, by the results of Section \ref{sec_bicyclic}, this gives a polynomial-time algorithm to check whether a given identity holds in the bicyclic monoid. (The work of Pastijn \cite{P06} also yields an algorithm for this problem in the bicyclic monoid, which also amounts to a reduction to linear programming; although Pastijn does not analyse the complexity of his algorithm, we believe it is similar to ours.)

Semigroups of tropical matrices are closely related to \textit{max-plus automata}, a computational model defined as a quantitative extension of finite automata. (\textit{Weighted automata} more generally were introduced by Sch{\"u}tzen-berger in \cite{Schutz61}; these machines perform an automatic computation on words over a finite alphabet, associating with each word a value or `weight' in a given semiring.) Semigroups of tropical matrices and max-plus automata provide two perspectives on the same underlying mathematical structures; the questions raised about them by algebraists and computer scientists are often different, but the study of identities is natural from both points of view; from a computational perspective it addresses the question of which pairs of distinct inputs can be distinguished by a computational model. This common ground has fruitfully been explored in the case of finite (unweighted) automata, leading to the equational and topological development of profinite theory \cite{Reiterman82, GGP08}.  In a companion article \cite{DJ17} the first two authors investigate the computational power of two-state max-plus automata, and use the ideas developed there to construct minimal length identities for $M_2(\trop)$.  A full understanding of all identities in $M_n(\trop)$ for $n \geq 3$ seems still quite distant.

\section{Preliminaries}
\label{sec_defs}

\subsection{Tropical polynomials}
\label{subsec_poly}
The \emph{tropical semiring} $\trop$ is the commutative idempotent semiring whose elements are drawn from the set $\mathbb{R} \cup \{-\infty\}$, and whose binary operations are defined by $a \oplus b = {\rm max} (a,b)$ and $a \otimes b = a+b$, where $-\infty$ should be thought of as the ``zero'' element of the semiring, satisfying $-\infty \oplus a = a$ and $-\infty \otimes a = -\infty$ for all $a \in \trop$. Often it will be convenient to work without this zero element; in order to distinguish this algebraic structure from the field $\mathbb{R}$ we shall write $\ft$ to denote the ``finitary tropical semiring'', namely the set $\mathbb{R}$, under the operations $\oplus$ and $\otimes$. There is an obvious total order on $\mathbb{T}$ in which $-\infty$ is the least element. We denote the supremum of a bounded-above set of elements $Y \subset \mathbb{T}$ by $\bigoplus_{y \in Y} y$.

By a \textit{formal tropical polynomial} in variables from a (countable) set $X$ we mean an element of the commutative polynomial semiring $\mathbb{T}[X]$, that is, a finite formal sum in which each term is a formal product of a coefficient from $\mathbb{FT}$ and formal powers of finitely many of the variables of $X$, considered up to the commutative and distributive laws and the idempotent addition in $\mathbb{T}$.  The order on $\mathbb{T}$ induces a partial order on $\mathbb{T}[X]$, in which $g \preceq f$ if and only if each term of $g$ appears as a term of $f$ with coefficient greater than or equal to the corresponding coefficient in $g$. A set of polynomials is \textit{bounded} if it is bounded above in this order; any bounded set of polynomials has a supremum which is a polynomial. We note that evaluation at any fixed values of the variables gives a supremum-preserving semiring morphism from $\trop[X]$ to $\trop$.

Each formal tropical polynomial naturally defines a function from $\mathbb{T}^X$ to $\mathbb{T}$, by interpreting all formal products and formal sums tropically. Unlike with classical polynomials over an infinite field, however, two distinct formal tropical polynomials may define the same function. For example, $x^{2} \oplus x \oplus 1$ and $x^{2} \oplus 1$ are distinct formal polynomials but define the same function because $x$ can never exceed both $x^{2}$ and $1$, as illustrated in Figure 1 below.

\begin{figure}[h!]
\includegraphics[width=0.6\textwidth]{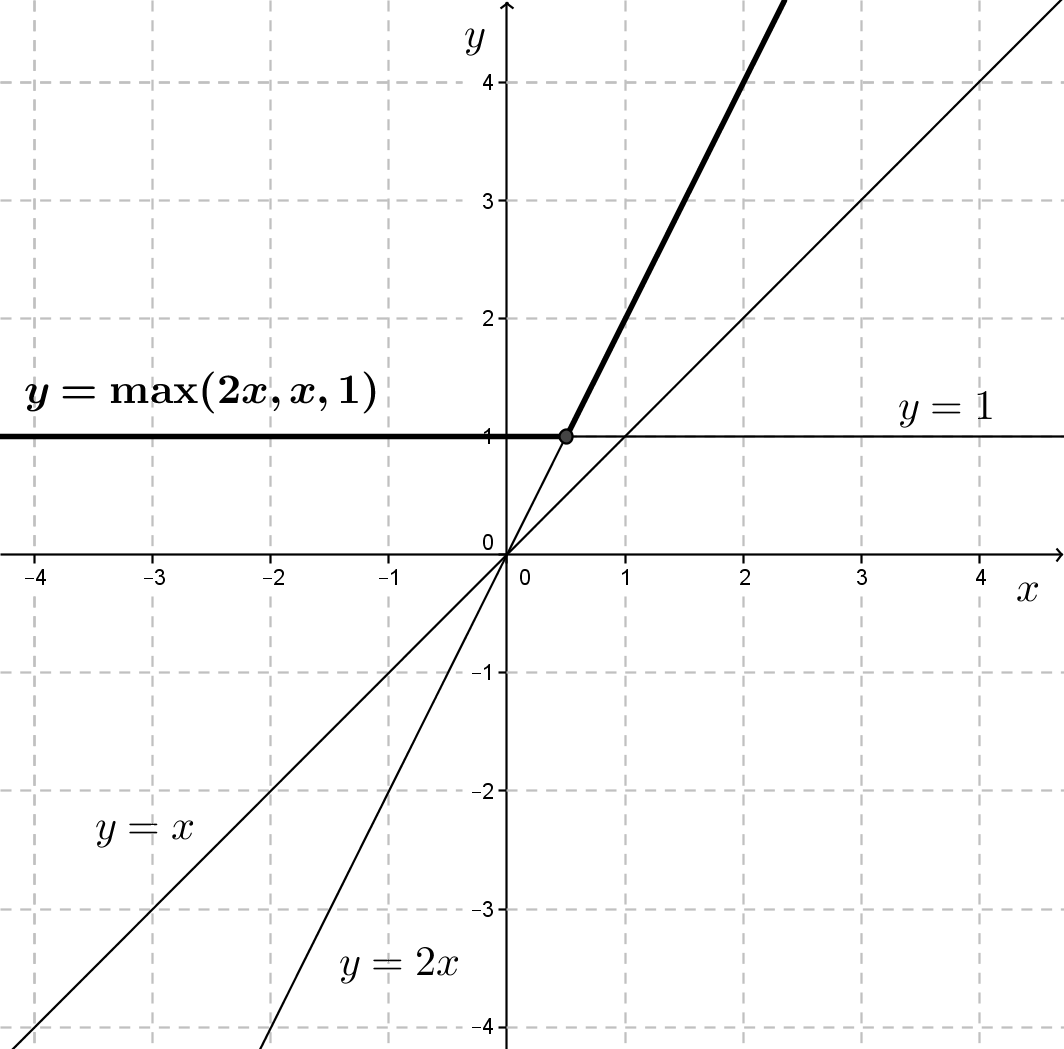}
\caption{The formal polynomial $x^2 \oplus x \oplus 1$ defines the function $\mathbb{T} \rightarrow \mathbb{T}, \ x \mapsto {\rm max}(2x, x, 1)$. Since $x$ cannot simultaneously exceed $2x$ and $1$, this is the same as the function $\mathbb{T} \rightarrow \mathbb{T}, \ x \mapsto {\rm max}(2x, 1)$ which corresponds to the formal polynomial $x^2 \oplus 1$.}
\end{figure}

We say that two formal polynomials are \textit{equivalent} if they represent the same function. We note that it is easy to see that two tropical polynomials $f, g: \mathbb{T}^X \rightarrow \mathbb{T}$ will be equivalent if and only if they agree on $\mathbb{FT}^X$.

We say that a tropical polynomial $f$ is \textit{$0$-flat} if it is a sum of products of variables (without coefficients). An elementary but important property of $0$-flat polynomials is the following relationship with classical scaling:
\begin{lemma}\label{lem:harmony}
If $f \in \mathbb{T}[x_1, \dots, x_k]$ is $0$-flat and $\lambda \in \mathbb{R}$ is non-negative then $f(\lambda x_1, \dots, \lambda x_k) = \lambda f(x_1, \dots, x_k)$ for all evaluations of $x_1, \dots, x_k$ in $\mathbb{T}$.
\end{lemma}
\begin{proof}
This is immediate from the fact that scaling by a non-negative value distributes over both maximum and classical addition.
\end{proof}

\subsection{Tropical semigroups}
\label{subsec_semigp}
It is easy to see that the set of all $n \times n$ matrices with entries in $\trop$ (respectively $\ft$) forms a monoid (respectively semigroup) under the matrix multiplication induced from the operations $\oplus$ and $\otimes$. We denote these semigroups by $M_n(\trop)$ and $M_n(\ft)$.  We also write $UT_n(\trop)$ (respectively $UT_n(\ft)$) to denote the subsemigroup of $M_n(\trop)$ consisting of those matrices in $M_n(\trop)$ with entries below the main diagonal all $-\infty$ and entries on and above the main diagonal drawn from $\trop$ (respectively $\ft$).

More generally, let $(\Gamma, \rho)$ be a (not necessarily finite) set equipped with a reflexive, transitive binary relation. Consider the set
$\trop^{\Gamma \times \Gamma}$ of functions from $\Gamma \times \Gamma$ to $\trop$; we think of its elements as
matrices with rows and columns indexed by $\Gamma$, and use the notation $M_{i,j}$ for $M(i,j)$. We say that $A \in \trop^{\Gamma \times \Gamma}$ is bounded above if the subset $\{A_{i,j}: i, j \in \Gamma\} \subseteq \mathbb{T}$ has an upper bound in $\mathbb{T}$. We define
\begin{eqnarray*}
\Gamma(\trop) &=&\{A \in \trop^{\Gamma \times \Gamma} \mid A \mbox{ is bounded above and } A_{i,j} \neq -\infty \implies i \rho j \},\\
\Gamma(\ft) &=&\{A \in \trop^{\Gamma \times \Gamma} \mid A \mbox{ is bounded above and } A_{i,j} \neq -\infty \iff i \rho j \}.
\end{eqnarray*}
We define a multiplication on $\Gamma(\trop)$ (and on $\Gamma(\ft)$) as for matrices by:
$$(M \otimes N)_{i,j} = \bigoplus_{k \in \Gamma} M_{i,k} \otimes N_{k,j}.$$
The boundedness of $M$ and $N$ ensures the product is defined and bounded, while the transitivity of $\Gamma$ ensures that appropriate entries of the product are $-\infty$. The operation is also easily seen to be associative, so it gives $\Gamma(\trop)$ and $\Gamma(\ft)$ the structure of semigroups.

Notice that taking $\Gamma = \lbrace 1, \dots, n \rbrace$ with the complete binary relation [respectively, the obvious partial order] we have
$\Gamma(\trop) = M_n(\trop)$ and $\Gamma(\ft) = M_n(\ft)$ [respectively, $\Gamma(\trop) = UT_n(\trop)$ and $\Gamma(\ft) = UT_n(\ft)$].
More generally, if $\Gamma$ is finite then fixing a bijection with a set $\lbrace 1, \dots, n \rbrace$ yields an embedding of $\Gamma(\trop)$ into
$M_n(\trop)$. If $\Gamma$ is a finite partial order then any order preserving bijection  from $\Gamma$ to $\{1, \ldots, n\}$ embeds $\Gamma(\trop)$ in some $UT_n(\trop)$.

\subsection{Semigroup identities}
\label{subsec_identities}
We write $\mathbb{N}_0$ and $\mathbb{N}$ respectively for the natural numbers with and without $0$. If $\Sigma$ is a finite alphabet, then $\Sigma^*$ will denote the free monoid on $\Sigma$, that is, the set of
finite (possibly empty) words over $\Sigma$ under the operation of concatenation. We write $\Sigma^+$ for the subsemigroup of non-empty words
in $\Sigma^*$, which is the free semigroup on $\Sigma$. For $w \in \Sigma^*$ and $s \in \Sigma$ we write $|w|$ for
the length of $w$ and $|w|_s$ for the number of occurrences of the letter $s$ in $w$. For $1 \leq i \leq |w|$ we write $w_i$ to denote the $i$th letter of $w$. The \textit{content} of $w$ is the map
$\Sigma \to \mathbb{N}_0, s \mapsto |w|_s$.

Recall that a (\textit{semigroup}) \textit{identity} is a pair of words, usually written ``$u=v$'', in the free semigroup $\Sigma^+$ on an alphabet $\Sigma$. We say that the identity \textit{holds} in a semigroup $S$ (or that $S$ \textit{satisfies} the identity) if every morphism from $\Sigma^+$ to $S$ maps $u$ and $v$ to the same element of $S$. If a morphism maps $u$ and $v$ to the same element we say that it \textit{satisfies} the given identity in $S$; otherwise
it \textit{falsifies} it. (\textit{Monoid} identities are defined similarly using $\Sigma^*$ in place of $\Sigma^+$; we shall work primarily with semigroup identities, but it is easy to deduce corresponding results about monoid identities from these.)

For $u, w \in \Sigma^+$, we say that $u$ is a \emph{scattered subword} of $w$ if there are words $w_1, \ldots w_{k+1}, u_1, \ldots, u_k \in \Sigma^*$ such that $w=w_1u_1w_2u_2 \cdots w_ku_kw_{k+1}$ and $u=u_1u_2\cdots u_k$.

For some purposes it is easier to work with matrices over $\trop$, while for others it is easier to work with matrices over $\ft$. The following proposition ensures that we can choose to work with whichever is easier at each stage.

\begin{proposition}\label{prop_finitarywilldoGamma} Let $(\Gamma, \rho)$ be a set equipped with a reflexive, transitive binary relation. Then $\Gamma(\trop)$ and $\Gamma(\ft)$ satisfy exactly the same semigroup identities.
\end{proposition}

\begin{proof}
Since $\Gamma(\ft)$ is a subsemigroup of $\Gamma(\trop)$, any identity satisfied in the latter is also satisfied in the former. To prove the converse
we shall work with a semigroup $\Gamma(\trop[x])$, which is constructed by exactly the same process as $\Gamma(\trop)$ but using the formal polynomial semiring $\trop[x]$ in place of $\trop$; its elements are thus functions $\Gamma \times \Gamma \to \trop[x]$ which are bounded above with respect to the order $\preceq$ on $\trop[x]$. For any evaluation of $x$ in $\mathbb{T}$, we obtain a supremum-preserving semiring morphism from $\trop[x]$ to $\trop$, which in turn it induces a semigroup morphism from $\Gamma(\trop[x])$ to $\Gamma(\trop)$.

Returning to the proof of the converse, suppose an identity $u=v$ over alphabet $\Sigma$ is \textit{not} satisfied in $\Gamma(\trop)$. Let $\phi : \Sigma^+ \to \Gamma(\trop)$ be a morphism falsifying the identity, and choose $i$ and $j$ with $\phi(u)_{i,j} \neq \phi(v)_{i,j}$.

Define a morphism $\psi : \Sigma^+ \to \Gamma(\trop[x])$ by for each $a \in \Sigma$ setting
$$\psi(a)_{k,l} = \begin{cases} \phi(a)_{k,l} & \textrm{ if $\phi(a)_{k,l} \neq -\infty$} \\
x &\textrm{ if $k \rho l$ but $\phi(a)_{k,l} = -\infty$} \\
-\infty &\textrm{ otherwise.}
\end{cases}$$
For any $x \in \trop$, because evaluation at $x$ induces a semigroup morphism from $\Gamma(\trop[x])$ to $\Gamma(\trop)$, we may define a semigroup morphism
$\psi_x : \Sigma^+ \to \Gamma(\trop)$ by $\psi_x(w)_{i,j} = [\psi(w)_{i,j}](x)$. If $x\in\ft$
then the image of $\psi_x$ is contained in $\Gamma(\ft)$, while if $x = -\infty$ we recover the original morphism $\phi$.

Consider now the formal polynomials $\psi(u)_{i,j}$ and $\psi(v)_{i,j}$. Since
at $x= -\infty$ they evaluate to $\phi(u)_{i,j}$ and $\phi(v)_{i,j}$ respectively, which are different, they must have different
constant terms. It follows that for $x$ small enough but finite, they evaluate to different values. Thus, for small enough
$x$, $\psi_x$ is a morphism to $\Gamma(\ft)$ falsifying the identity $u=v$.
\end{proof}
Proposition \ref{prop_finitarywilldoGamma} has the following immediate consequence:
\begin{corollary}\label{cor_finitarywilldo} For all $n \in \mathbb{N}$,
\begin{itemize}
\item[(i)] $M_n(\trop)$ and $M_n(\ft)$ satisfy exactly the same semigroup identities; and
\item[(ii)] $UT_n(\trop)$ and $UT_n(\ft)$ satisfy exactly the same semigroup identities.
\end{itemize}
\end{corollary}

The following two elementary results will allow us at various points to simplify the structure of and number of parameters in the matrices we must consider.

\begin{lemma}\label{lemma_independentscale}
Let $\phi : \Sigma^+ \to M_n(\trop)$ be a morphism, $\mu_s \in \ft$
for each $s \in \Sigma$, and define a new morphism $\psi : \Sigma^+ \to M_n(\trop)$ by $\psi(s) = \mu_s \otimes \phi(s)$ for all $s \in \Sigma$.
If $u$ and $v$ are words with the same content and $\phi(u) = \phi(v)$, then $\psi(u) = \psi(v)$.
\end{lemma}
\begin{proof}
Since tropical scaling commutes with matrix multiplication it is easy to see that
$$\psi(u) \ = \ \left(\bigotimes_{s \in \Sigma} \mu_s^{|u|_s}\right) \otimes \phi(u) \ = \ \left(\bigotimes_{s \in \Sigma} \mu_s^{|v|_s}\right) \otimes \phi(v) \ = \ \psi(v),$$
where the powers are to be interpreted tropically, and the middle equality holds because $\phi(u) = \phi(v)$ and $u$ and $v$ have the same content.
\end{proof}

\begin{lemma}\label{lemma_bottomright}
Suppose an identity $u = v$ over alphabet $\Sigma$ is satisfied by all morphisms $\phi : \Sigma^+ \to UT_n(\ft)$ with the property
that $\phi(s)_{n,n} = 0$ for all $s \in \Sigma$. If $u$ and $v$ have the same content then the identity $u = v$ holds in $\utnt$.
\end{lemma}
\begin{proof}
Every morphism to $UT_n(\ft)$ can clearly be obtained from one of the given form by the construction in
Lemma~\ref{lemma_independentscale}, which since $u$ and $v$ have the same content means it satisfies the identity.
So $UT_n(\ft)$ and hence also (by Corollary~\ref{cor_finitarywilldo}) $UT_n(\trop)$ satisfy the identity.
\end{proof}

\section{Identities in $UT_2(\trop)$.}\label{sec_2x2identities}

In this section we shall give an elementary but powerful characterisation of the identities which hold in $UT_2(\trop)$ (or equivalently,
by Corollary~\ref{cor_finitarywilldo}, in $UT_2(\ft)$).

Let $w$ be a word over an alphabet $\Sigma$. For $s \in \Sigma$ and $0 \leq i \leq |w|$ we write $\lambda_s^w(i)$ for the number
of occurrences of the letter $s$ within the first $i$ letters of the word $w$ (so in particular $\lambda_s^w(0) = 0$ for all $s \in \Sigma$).
For each $t \in \Sigma$ we define a $0$-flat formal tropical polynomial having variables $x_s$ for each $s \in \Sigma$ as follows:
$$f_t^w=  \bigoplus_{w_i = t} \bigotimes_{s \in \Sigma} x_s^{\lambda_s^w(i-1)},$$
where of course the powers are to be interpreted tropically, and a maximum over the empty set is taken to be $-\infty$. By a slight abuse of notation, given $\underline{x} \in \mathbb{FT}^{\Sigma}$ we shall also write $x_s$ to denote the image of $s$ under $\underline{x}$, and $f_t^w(\underline{x})$ to denote the corresponding evaluation of $f_t^w$ at the variables $x_s \in \mathbb{FT}$.

We shall see shortly (Theorem~\ref{thm_kletter}) that the values of the polynomials $f_t^w$ exactly characterise the identities which
hold in $UT_2(\trop)$. The following lemma explains how they arise from the multiplication in $UT_2(\ft)$:

\begin{lemma}\label{lemma_topright}
Suppose a morphism $\phi : \Sigma^+ \to UT_2(\trop)$ is such that $\phi(s)_{1,1} \in \ft$ and $\phi(s)_{2,2} = 0$ for all $s \in \Sigma$,
say
$$\phi(s) = \left(\begin{array}{c c}
x_s & x_s'\\
-\infty& 0
\end{array}\right), \mbox{ where } x_s, x_s' \in \mathbb{T} \mbox{ with } x_s \neq -\infty.$$
Then for any word $w \in \Sigma^+$,
$$\phi(w)_{1,2} = \bigoplus_{s \in \Sigma} x_s' \otimes f_s^w(\ul{x})$$
\end{lemma}
\begin{proof}
This is a straightforward calculation, using the definitions of matrix multiplication and of the polynomial functions $f_s^w$.
\end{proof}

The next lemma says that the polynomial functions $f_t^w$ are together sufficient to characterise the content of the word $w$.

\begin{lemma}\label{lemma_samecontent}
If $f_t^w(\underline{x})=f_t^v(\underline{x})$ for all $t \in \Sigma$ and $\underline{x} \in \mathbb{FT}^\Sigma$ then $w$ and $v$ have the same content.
\end{lemma}
\begin{proof}
It is easy to see that $f_t^w$ is the constant $-\infty$ function if and only if $w$ does not contain any occurrences of the letter $t$. Thus, we may
assume that exactly the same letters from $\Sigma$ occur in $w$ and $v$.

Let $z \in \Sigma$, and suppose that $z$ occurs in both $w$ and $v$.
Define $\ul{x} \in \mathbb{FT}^\Sigma$ by $x_z = 1$ and $x_s = 0$ for $s \in \Sigma \setminus \lbrace z \rbrace$. Then
$$f_z^w(\ul{x}) \ = \ \bigoplus_{w_i = z} \bigotimes_{s \in \Sigma} x_s^{\lambda_s^w(i-1)}
\ = \ \bigoplus_{w_i = z} \lambda_z^w(i-1)$$
where the last equality is because of the values of $\underline{x}$. Since for fixed $z$ and $w$ the function $\lambda_z^w(i)$ is non-decreasing with $i$, it is clear that the maximum on the right is attained when $i$ is the position of the final $z$ in $w$. With this value of $i$, we see that $f_z^w(x) = \lambda_z^w(i-1) = |w|_z -1$. By the same argument, $f_z^v(x) = |v|_z - 1$. But by assumption $f_z^w(x) = f_z^v(x)$, so we must have $|w|_z = |v|_z$.
\end{proof}

We are now ready to prove the main theorem of this section.

\begin{theorem}
\label{thm_kletter}
The identity $w=v$ over alphabet $\Sigma$ is satisfied in $UT_2(\trop)$ if and only if the tropical polynomials $f_t^w$ and $f_t^v$ are equivalent for all
$t \in \Sigma$.
\end{theorem}

\begin{proof}
Suppose first that we can choose $\underline{x} \in \mathbb{FT}^\Sigma$ and $t \in \Sigma$, such that  $f_t^w(\underline{x})\neq f_t^v(\underline{x})$. Define a morphism
$$\phi \ : \ \Sigma^+ \to \uttt, \ \ \ a \mapsto \left(\begin{array}{c c}
x_a & x_a'\\
-\infty& 0
\end{array}\right)$$
where $x_t' = 0$ and $x_s' = -\infty$ for each $s \neq t$. Then by Lemma~\ref{lemma_topright},
$$\phi(w)_{1,2} \ = \ \bigoplus_{s \in \Sigma} x_s' \otimes f_s^w(\ul{x}) \ = \ f_t^w(\ul{x})$$
and similarly
$$\phi(v)_{1,2} \ = \ \bigoplus_{s \in \Sigma} x_s' \otimes f_s^v(\ul{x}) \ = \ f_t^v(\ul{x}) \ \neq \ \phi(w)_{1,2}$$
so the morphism $\phi$ falsifies the identity in $\uttt$.

Conversely, suppose that $f_t^w(\underline{x})=f_t^v(\underline{x})$ for all $t \in \Sigma$ and all $\underline{x} \in \mathbb{FT}^\Sigma$. By
Lemma~\ref{lemma_samecontent}, $w$ and $v$ have the same content. Hence, by Lemma~\ref{lemma_bottomright} it suffices to show that the identity $w=v$ is satisfied by morphisms $\phi : \Sigma^+ \to UT_2(\ft)$ such that $\phi(s)_{2,2} = 0$ for all $s \in S$. Let $\phi$ be such a morphism, and for each $s \in \Sigma$
write $x_s = \phi(s)_{1,1}$ and $x_s' = \phi(s)_{1,2}$.

It is immediate from the properties of $\phi$ that $\phi(w)_{2,2} = 0 = \phi(v)_{2,2}$.
It is also easy to see that for any word $u$,
$$\phi(u)_{1,1} \ = \ \bigotimes_{i=1}^{|u|} \phi(u_i)_{1,1} \ = \ \sum_{s \in \Sigma} |u|_s \phi(s)_{1,1}.$$
Since $w$ and $v$ have the same content, it follows that $\phi(w)_{1,1} = \phi(v)_{1,1}$.
Finally, Lemma~\ref{lemma_topright} gives
$$\phi(w)_{1,2} \ = \ \bigoplus_{t \in \Sigma} x_t' \otimes f_t^w(\ul{x}) \ = \ \bigoplus_{t \in \Sigma} x_t' \otimes f_t^v(\ul{x}) \ = \ \phi(v)_{1,2}$$
so that $\phi(w) = \phi(v)$ as required.
\end{proof}

We conclude this section by noting that checking whether a pair of words on an alphabet containing just two letters defines an identity on $UT_2(\mathbb{T})$ amounts to a comparison of four pairs of tropical polynomials in just one variable.
\begin{corollary}\label{cor_kletter}
\label{2letterreduced}
Let $\Sigma$ be a two-letter alphabet and $w,v \in \Sigma^+$. Then the identity $w=v$ holds in $UT_2(\mathbb{T})$ if and only if $f_z^w(x,1)=f_z^v(x,1)$ and $f_z^w(x,-1)=f_z^v(x,-1)$ for all $x \in \mathbb{FT}$, and all $z \in \Sigma$.
\end{corollary}

\begin{proof}
The direct implication follows immediately from Theorem~\ref{thm_kletter}. Conversely, suppose that $f_z^w(x,1)=f_z^v(x,1)$ and $f_z^w(x,-1)=f_z^v(x,-1)$ for all $z \in \Sigma$ and all $x \in \mathbb{FT}$. Since the polynomials $f_z^w$ and $f_z^v$ are $0$-flat, for any $c,d \in \mathbb{R}$ with $d \neq 0$ applying Lemma~\ref{lem:harmony} we have
$$f_z^w(c,d) \ = \ |d| f_z^w\left( \frac{c}{|d|}, \frac{d}{|d|} \right) \ = \ |d| f_z^v \left( \frac{c}{|d|}, \frac{d}{|d|} \right) \ = \ f_z^v(c,d)$$
where $|d|$ denotes the absolute value of $d$ so that $d / |d|$ is $1$ or $-1$.
Since the functions of the form $f_z^w(c, -), f_z^v(c, -): \mathbb{R} \rightarrow \mathbb{R}$ are continuous it follows also that
$f_z^w(c,0) = f_z^v(c,0)$ for all $c \in \mathbb{R}$, so by Theorem~\ref{thm_kletter} the identity $w=v$ holds in $UT_2(\mathbb{T})$.
\end{proof}

\begin{example}\label{example_identity}  Consider the words
$$w=AB^2AABBA^2B \textrm{ and } v=AB^2ABABA^2B$$
over the two letter alphabet $\Sigma = \{A,B\}$.
The identity $w=v$ is well-known to hold in the bicyclic monoid (see \cite{Dja77,Shl90} and Section~\ref{sec_bicyclic} below); we claim that it is also satisfied in $UT_2(\mathbb{T})$. The relevant tropical polynomials are
\begin{eqnarray*}
f_A^w &=& x_A^0x_B^0 \oplus x_A^1x_B^2 \oplus x_A^2x_B^2 \oplus x_A^3x_B^4 \oplus x_A^4x_B^4,\\
f_A^v &=& x_A^0x_B^0 \oplus x_A^1x_B^2 \oplus x_A^2x_B^3 \oplus x_A^3x_B^4 \oplus x_A^4x_B^4,\\
f_B^w &=& x_A^1x_B^0 \oplus x_A^1x_B^1 \oplus x_A^3 x_B^2 \oplus x_A^3 x_B^3 \oplus x_A^5x_B^4, \textrm{ and} \\
f_B^v &=& x_A^1x_B^0 \oplus x_A^1x_B^1 \oplus x_A^2x_B^2 \oplus x_A^3x_B^3 \oplus x_A^5x_B^4.
\end{eqnarray*}
By Corollary \ref{cor_kletter} it suffices to check whether
$f_A^w(x,1) =  f_A^v(x, 1)$, $f_A^w(x,-1) =  f_A^v(x,-1)$, $f_B^w(x, 1) =  f_B^v(x, 1)$ and $f_B^w(x,-1) =  f_B^v(x,-1)$ for
all $x \in \ft$.
Expanding in classical notation the equations to check become:
\begin{eqnarray*}
{\rm max} (0,x+2, \underline{2x+2},3x+4, 4x+4) &=& {\rm max} (0,x+2, \underline{2x+3},3x+4, 4x+4),\\
{\rm max} (0,x-2, \underline{2x-2},3x-4, 4x-4) &=& {\rm max} (0,x-2, \underline{2x-3},3x-4, 4x-4),\\
{\rm max} (x,x+1, \underline{3x+2},3x+3, 5x+4) &=& {\rm max} (x,x+1, \underline{2x+2},3x+3, 5x+4),\\
{\rm max} (x,x-1, \underline{3x-2},3x-3, 5x-4) &=& {\rm max} (x,x-1, \underline{2x-2},3x-3, 5x-4).
\end{eqnarray*}
In each equation we have underlined those terms which occur on one side but not the other, and it is straightforward to see that these terms never
affect the value of the maximum. For example, in the right-hand-side of the first equation, $2x+3$ cannot simultaneously strictly exceed both $x+2$ and $3x+4$.
\end{example}

\section{The Bicyclic Monoid}\label{sec_bicyclic}

The \textit{bicyclic monoid}  $\mathcal{B} = \langle p,q \mid pq=1 \rangle$ is an inverse monoid which is ubiquitous in almost all areas of infinite semigroup theory, and indeed in many other areas of mathematics. Identities in the bicyclic monoid have been extensively studied; Adjan~\cite{A66} established that the identity $AB^2A^2BAB^2A= AB^2A BA^2B^2A$ is a minimal length identity for this monoid.  Shleifer~\cite{Shl90} has shown (with computer assistance) that, up to relabelling of letters and exchanging which word appears on the left or right of the equality, there is just one other identity of minimal length: namely the one considered in Example~\ref{example_identity} above. Jones~\cite{J87} has shown that a semigroup variety contains a simple semigroup that is not bisimple if and only if it contains the bicyclic semigroup. Motivated by this work, Pastijn~\cite{P06} gave an elegant geometric characterisation of the identities which hold in the bicyclic monoid in terms of the properties of associated polyhedral complexes. Most recently Shneerson \cite{Shn15} has provided a family of interesting examples of semigroups of cubic growth which satisfy the same identities as the bicyclic monoid.

In this section we apply Theorem~\ref{thm_kletter} to prove that the identities satisfied by $UT_2(\trop)$ are precisely those satisfied by the bicyclic monoid, thus establishing that $UT_2(\trop)$ and several other associated semigroups all generate the same variety as $\mathcal{B}$.

\begin{theorem}
\label{main_theorem}
The bicyclic monoid $\mathcal{B}$ and the upper triangular tropical matrix monoid $UT_2(\trop)$ satisfy exactly the same semigroup identities.
\end{theorem}

\begin{proof}
It is well known and easy to see that every element of the bicyclic monoid can be written uniquely in the form $q^i p^j$ for some $i, j \in \mathbb{N}_0$.
Izhakian and Margolis \cite{IM10} noted that the map:
$$\rho \ \ : \ \ \mathcal{B} \to UT_2(\trop), \ \ q^i p^j \mapsto \left(\begin{array}{c c}
 \label{bicyclicmat}
i-j& i+j\\
-\infty& j-i
\end{array}\right)$$
is a semigroup embedding of $\mathcal{B}$ into $UT_2(\trop)$; of course, it also gives an embedding into $UT_2(\ft)$.

It is immediate from the existence of this embedding that any identity which holds in $UT_2(\trop)$ must hold in $\mathcal{B}$. It remains to show that any identity which holds in $\mathcal{B}$ must also hold in $UT_2(\trop)$. We shall show the contrapositive.
To this end, suppose that a semigroup identity $w = v$ over finite alphabet $\Sigma$ does not hold in $UT_2(\trop)$.
Clearly we may assume that $w$ and $v$ have the same content; indeed if $z \in \Sigma$ is such that $|w|_z \neq |v|_z$ then the identity is falsified in $\mathcal{B}$ by (for example) the morphism sending $z$ to $p$ and every other letter to the identity element.

By Theorem~\ref{thm_kletter} there exists
$\underline{x} \in \mathbb{FT}^\Sigma$ and $t \in \Sigma$ such that $f_t^w(\underline{x}) \neq f_t^v(\underline{x})$. Suppose without loss of generality that $f_t^w(\underline{x}) > f_t^v(\underline{x})$. Since both functions are piecewise linear in the components of $\underline{x}$, there is an open neighbourhood (in the usual Euclidean topology on $\mathbb{FT}^\Sigma$ considered as a set of finite dimensional vectors over $\mathbb{R}$) of $\underline{x}$ in which every vector satisfies the given inequality. This neighbourhood must contain a vector in $\mathbb{Q}^\Sigma$ so by replacing $\underline{x}$ with such a vector, we may assume without loss of generality that the components of $\underline{x}$ are all rational.

Let $d$ denote the least common multiple of the denominators in the lowest forms of the entries $x_s$ for $s \in \Sigma$. Since $f_t^w$ and $f_t^v$ are $0$-flat, it follows from Lemma \ref{lem:harmony} that $f_t^w(2d\underline{x}) > f_t^v(2d\underline{x})$. Hence, by replacing $\ul{x}$ with $2d\underline{x}$, we may further assume, again without loss of generality, that the components of $\underline{x}$ are even integers.

Now for each $s \in \Sigma$, choose a non-negative even integer $x_s'$ greater than $x_s$, in such a way that $x_t'$ is very large relative to all
other $x_s'$. Define a morphism $\psi : \Sigma^+ \to UT_2(\ft)$ by
$$\psi(s) = \left(\begin{array}{c c}
x_s & x_s'\\
-\infty& 0
\end{array}\right) \mbox{ for all } s \in \Sigma.$$
By Lemma~\ref{lemma_topright} we have
$$\psi(w)_{1,2} \ = \ \bigoplus_{s \in \Sigma} x_s' \otimes f_s^w(\ul{x}) \ = \ x_t' \otimes f_t^w(\ul{x}),$$
provided $x_t'$ was chosen sufficiently large, and similarly
$$\psi(v)_{1,2} \ = \ \bigoplus_{s \in \Sigma} x_s' \otimes f_s^v(\ul{x}) \ = \ x_t' \otimes f_t^v(\ul{x}) \ \neq \ \psi(w)_{1,2},$$
where the inequality is because $f_t^w(\underline{x}) \neq f_t^v(\underline{x})$.
So the morphism $\psi$ falsifies the identity in $UT_2(\ft)$.

Next for each $s \in \Sigma$ let $i_s = \frac{1}{2} x_s'$, and $j_s = \frac{1}{2}(x_s'-x_s)$. Then $i_s$ is non-negative (because $x_s'$ is non-negative) and
integer (because $x_s'$ is even), while $j_s$ is non-negative (because $x_s' > x_s$) and integer (because $x_s'$ and $x_s$ are even).
Define a morphism $\phi : \Sigma^+ \to \operatorname{im} \rho \subseteq \uttt$ by
$$\phi(s) \ = \ \rho(q^{i_s} p^{j_s}) =
\left(\begin{array}{c c}
i_s - j_s & i_s + j_s \\
-\infty& j_s - i_s
\end{array}\right) \ = \ (j_s - i_s) \otimes \left( \begin{array}{cc}
x_s & x_s' \\
-\infty& 0
\end{array}\right)$$
where the last equality follows from the definitions of $i_s$ and $j_s$. Because of the last equality, we see that $\psi$ can be obtained from
$\phi$ by the construction in Lemma~\ref{lemma_independentscale}, with $\mu_s = i_s - j_s$ for each $s$. Hence, by the contrapositive of Lemma~\ref{lemma_independentscale}, $\phi$ falsifies the identity.

But $\operatorname{im} \phi \subseteq \operatorname{im} \rho$ which is isomorphic to $\mathcal{B}$, so the identity
cannot hold in $\mathcal{B}$.
\end{proof}

\begin{corollary}
The subsemigroups of $\uttt$ obtained by restricting the on-and-above diagonal entries to be respectively in $\mathbb{Q}$, $\mathbb{Z}$, $\mathbb{N}_0$,
$\mathbb{Q} \cup \lbrace -\infty \rbrace$, $\mathbb{Z} \cup \lbrace -\infty \rbrace$ and $\mathbb{N}_0 \cup \lbrace -\infty \rbrace$ all satisfy
exactly the same identities as $\mathcal{B}$ and $\uttt$.
\end{corollary}
\begin{proof}
Write $UT_2(X)$ for the given semigroups, where $X = \mathbb{Q}$, $\mathbb{Z}$, $\mathbb{N}_0$, $\mathbb{Q} \cup \lbrace -\infty \rbrace$, $\mathbb{Z} \cup \lbrace -\infty \rbrace$ or $\mathbb{N}_0 \cup \lbrace -\infty \rbrace$.
The semigroups $UT_2(\mathbb{Q})$, $UT_2(\mathbb{Z})$, $UT_2(\mathbb{Q} \cup \lbrace -\infty \rbrace)$ and $UT_2(\mathbb{Z} \cup \lbrace -\infty \rbrace)$ are all contained in $\uttt$ and contain the image of the embedding $\rho$, which is isomorphic to $\mathcal{B}$.

The semigroups $UT_2(\mathbb{N}_0)$ and $UT_2(\mathbb{N}_0 \cup \lbrace -\infty \rbrace)$ are again contained in $\uttt$, and so satisfy all of the given identities. Conversely, notice that every matrix in $UT_2(\mathbb{Z})$
is a tropical scaling of a matrix in $UT_2(\mathbb{N}_0)$. By the previous paragraph, any identity which does not hold in $\uttt$ is falsified by some morphism to $UT_2(\mathbb{Z})$, and now the contrapositive of
Lemma~\ref{lemma_independentscale} yields a morphism falsifying the identity with image contained in $UT_2(\mathbb{N}_0)$, and hence
also in $UT_2(\mathbb{N}_0 \cup \lbrace -\infty \rbrace)$.
\end{proof}

The bicyclic monoid is well known (see for example \cite[Section 1.6]{H95}) to be isomorphic to $\mathbb{N}_0 \times \mathbb{N}_0$ under
the multiplication
$$(a,b)(c,d) \ = \ (a-b+{\rm max}(b,c),d-c+{\rm max}(b,c)).$$
Several natural extensions of this construction  have been considered (see for example \cite{G08, TaylorThesis, W68}). We write $\mathcal{B}_{\mathbb{Z}}:=\mathbb{Z}\times \mathbb{Z}$, $\mathcal{B}_{\mathbb{Q}}:=  \mathbb{Q}\times \mathbb{Q}$ and $\mathcal{B}_{\mathbb{R}}=\mathbb{R}\times \mathbb{R}$ to denote the semigroups with completely analogous multiplication, noting the obvious embeddings $\mathcal{B} \subseteq \mathcal{B}_{\mathbb{Z}} \subseteq \mathcal{B}_{\mathbb{Q}} \subseteq \mathcal{B}_{\mathbb{R}}$.

\begin{corollary}\label{cor_bicyclic}
The semigroups $\mathcal{B}_{\mathbb{Z}}, \mathcal{B}_{\mathbb{Q}} $and  $\mathcal{B}_{\mathbb{R}}$ satisfy exactly the same semigroup identities as $\mathcal{B}$ and $\uttt$.
\end{corollary}
\begin{proof}
It is straightforward to check that the map
$$\phi \ : \ \mathcal{B}_{\mathbb{R}} \rightarrow UT_2(\trop), \ (a,b) \mapsto \left(\begin{array}{c c}
a-b & a+b \\
-\infty& b-a
\end{array}\right)$$
is an embedding. Thus we have embeddings
$$\mathcal{B} \hookrightarrow \mathcal{B}_{\mathbb{Z}} \hookrightarrow \mathcal{B}_{\mathbb{Q}} \hookrightarrow \mathcal{B}_{\mathbb{R}}\hookrightarrow UT_2(\trop).$$
Since $\mathcal{B}$ and $UT_2(\trop)$ satisfy the same identities, the result follows.
\end{proof}

\section{$UT_n(\mathbb{T})$ and Chain-Structured Tropical Matrix Semigroups}
\label{sec_chain}

In this section we extend the results of Section \ref{sec_2x2identities} to a more general class of semigroups $\Gamma(\trop)$ and $\Gamma(\ft)$ (including $UT_n(\mathbb{T})$ and $UT_n(\ft)$ for $n > 2$). We say that  $\Gamma(\trop)$ and $\Gamma(\ft)$ are \textit{chain-structured tropical matrix semigroups} if $\Gamma$ is a partial order with an upper bound on the length of chains. For the rest of this section, we assume that $\Gamma$ has these properties, and let $n$ be the length of the longest chain in $\Gamma$. (Thus there exists a sequence $v_1 \leq v_2 \leq \cdots \leq v_k$ of $k$ distinct elements of $\Gamma$ if and only if $k \leq n$.)

By a \textit{$k$-vertex walk} (or \textit{walk of vertex length} $k$) in $\Gamma$ we mean a $k$-tuple $(v_1, \dots, v_k)$ such that $v_1 \leq v_2 \leq \dots \leq v_k$. A \textit{$k$-vertex path} (or \textit{path of vertex length $k$}) is a $k$-vertex walk in which consecutive vertices (and hence all vertices) are distinct.

Let $w$ be a word over the alphabet $\Sigma$. For $0 \leq p < q \leq |w|+1$ and $s \in \Sigma$ we
define
$$\beta_s^w(p, q) \ = \ | \lbrace i \in \mathbb{N} \mid p < i < q, w_i = s \rbrace |$$
to be the number of occurrences of $s$ lying strictly \emph{between} $w_{p}$ and $w_{q}$. For each $u \in \Sigma^*$ with $|u| \leq n-1$ and each $(|u|+1)$-vertex path $\rho=(\rho_0, \rho_1, \ldots, \rho_{|u|})$ in $\Gamma$, we define a ($0$-flat) formal tropical polynomial having variables $x(s,i)$ for each letter $s \in \Sigma$ and vertex $i \in \Gamma$ as follows:
\begin{eqnarray*}
f_{u, \rho}^{w} = \ \bigoplus \bigotimes_{s\in \Sigma}\bigotimes_{k=0}^{|u|} x(s, \rho_k)^{\beta_s^w(\alpha_k, \alpha_{k+1})},
\end{eqnarray*}
where the sum ranges over all $0=\alpha_0< \alpha_1<\cdots<\alpha_{|u|}<\alpha_{|u|+1}=|w|+1$ such that $w_{\alpha_k}=u_k$ for $k=1, \ldots, |u|$. Here the powers are to be interpreted tropically, and a maximum taken over the empty set is taken to be $-\infty$. Thus it is easy to see that $f_{u, \rho}^{w} \neq -\infty$ if and only if $u$ is a scattered subword of $w$ of length equal to the path $\rho$. Note that taking $u$ to be the empty word forces $\rho = (\rho_0)$ for some $\rho_0 \in \Gamma$ and hence $f_{u, \rho}^w = \bigotimes_{s \in \Sigma} x(s, \rho_0)^{|w|_s}$, which is completely determined by the content of $w$.

In general, it is clear that the number of choices for the $\alpha_i$'s is bounded above by $|w|^{|u|}$, which in turn is bounded above by $|w|^{n-1}$. Thus, fixing $\Gamma$ (and hence $n$), it is easy to verify that the number of terms in the formal tropical polynomial $f_{u, \rho}^{w}$ is polynomial in $|w|$; this will be important in Section \ref{sec_complexity}.

\begin{lemma}\label{lemma_entriesGamma}
Let $\phi : \Sigma^+ \to \Gamma(\trop)$ be a morphism, and define $\underline{x} \in \mathbb{T}^{\Sigma \times \Gamma}$ by
$$\underline{x}(s,i) \ = \ \phi(s)_{i,i}.$$
Then for any word $w \in \Sigma^+$ and vertices $i, j \in \Gamma$ we have
\begin{equation}\label{eq1}
\phi(w)_{i,j} \ = \ \bigoplus_{\substack{u \in \Sigma^*,\\ |u| \leq n-1}}\bigoplus_{\rho \in \Gamma_{i,j}^{|u|}} \left(\bigotimes_{k=1}^{|u|} \phi(u_k)_{\rho_{k-1}, \rho_k}\right) \otimes f_{u, \rho}^w(\underline{x}),
\end{equation}
where $\Gamma^{|u|}_{i,j}$ denotes the set of all $(|u|+1)$-vertex paths from $i$ to $j$ in $\Gamma$.
\end{lemma}
\begin{proof}
Let $i$ and $j$ be vertices. Using the definition of the functions $f_{u,\rho}^w$, the value given to $\underline{x}$ and the distributivity of multiplication $\otimes$  over addition $\bigoplus$, the right-hand-side of \eqref{eq1} is equal to
\begin{align*}
\bigoplus_{\substack{u \in \Sigma^*,\\ |u| \leq n-1}} \bigoplus
\bigoplus_{\rho \in \Gamma_{i,j}^{|u|}} \left(\bigotimes_{k=1}^{|u|} \phi(u_k)_{\rho_{k-1}, \rho_k}\right) \otimes
\left( \bigotimes_{s\in \Sigma}\bigotimes_{k=0}^{|u|} (\phi(s)_{\rho_k, \rho_k})^{\beta_s^w(\alpha_k, \alpha_{k+1})} \right)
\end{align*}
where the unlabelled sum ranges over all $0=\alpha_0< \alpha_1<\cdots<\alpha_{|u|}<\alpha_{|u|+1}=|w|+1$ such that $w_{\alpha_k}=u_k$ for $k=1, \ldots, |u|$.
Notice that we are summing over all possible words $u$ of length less than $n$, and then over all scattered subwords of $w$ equal to $u$. Thus, we are simply summing
over all scattered subwords of $w$ of length less than $n$, so the above is equal to:
\begin{align*}
\bigoplus \bigoplus_{\rho \in \Gamma_{i,j}^l} \left(\bigotimes_{k=1}^{l} \phi(w_{\alpha_k})_{\rho_{k-1}, \rho_k}\right) \otimes
\left( \bigotimes_{s\in \Sigma}\bigotimes_{k=0}^{l} (\phi(s)_{\rho_k, \rho_k})^{\beta_s^w(\alpha_k, \alpha_{k+1})} \right)
\end{align*}
where the unlabelled sum is over all $0=\alpha_0< \alpha_1<\cdots<\alpha_{l+1}=|w|+1$ for some $l \leq n-1$
and we set $u = w_{\alpha_0} \dots w_{\alpha_l}$.

Now to each term in the above sum, defined by a choice of $\alpha_i$'s and a $\rho \in \Gamma_{i,j}^l$, we can associate
a $(|w|+1)$-vertex walk $(\sigma_0 = i, \dots, \sigma_{|w|} = j)$ in $\Gamma$ whose underlying path is $\rho$ and which transitions to vertex $\rho_k$ after $\alpha_{k}$ steps. Clearly every $(|w|+1)$-vertex walk from $i$ to $j$ arises exactly once in this way, and so effectively we are summing over all such walks.
In each term, the content of the left-hand parentheses gives a factor $\phi(w_q)_{\sigma_{q-1},\sigma_{q}}$ when $q = \alpha_k$ for some $k$,
while from the definition of the functions $\beta_s^w$, the content of the right-hand parentheses gives a factor $\phi(w_q)_{\sigma_{q-1},\sigma_q}$ for each $q$ not of this form. Thus, the above is simply equal to:
$$\bigoplus \bigotimes_{q=1}^{|w|} \phi(w_q)_{\sigma_{q-1},\sigma_q}$$
where the supremum is taken over all $(|w|+1)$-vertex walks $(i = \sigma_0, \sigma_1, \dots, \sigma_{|w|} = j)$ in $\Gamma$. But by the
definition of multiplication in $\Gamma(\trop)$, this is easily seen to be equal to $\left( \phi(w_1)\otimes \dots \otimes \phi(w_{|w|}) \right)_{i,j} \ = \ \phi(w)_{i,j}.$
\end{proof}

We are now ready to prove the main theorem of this section.

\begin{theorem}
\label{thm_kletterGamma}
Let $\Gamma$ be a partial order with maximum chain length $n$. Then the identity $w=v$ over alphabet $\Sigma$ is satisfied in $\Gamma(\trop)$ if and only if for every $u \in \Sigma^*$ with $|u| \leq n-1$ and every path $\rho$ of length $|u|$ in $\Gamma$ the tropical
polynomials $f_{u, \rho}^w$ and $f_{u, \rho}^v$ are equivalent.
\end{theorem}

\begin{proof}
Suppose first that $f_{u,\rho}^w(\underline{x})\neq f_{u,\rho}^v(\underline{x})$ for some choice of $u \in \Sigma^{+}$, $\rho$ (from $i$ to $j$) and $\underline{x} \in \mathbb{FT}^{\Sigma \times \Gamma}$. Define a morphism $\phi : \Sigma^+ \to \Gamma(\trop)$ by
$$\phi(s)_{p,p} \ = \ \underline{x}(s,p) \in \mathbb{FT}, \mbox{ for all } p \in \Gamma \mbox{ and } s \in \Sigma; \mbox{ and}$$
$$\phi(s)_{p,q} \ = \ \begin{cases}
0 & \mbox{if } s=u_i, p=\rho_{i-1}, q=\rho_i\\
-\infty & \mbox{otherwise}.
\end{cases}$$
Then by Lemma~\ref{lemma_entriesGamma},
$$\phi(v)_{i,j} \ = \ f_{u, \rho}^v(\ul{x}) \ \neq \ f_{u, \rho}^w(\ul{x}) \ = \ \phi(w)_{i,j},$$
and so the morphism $\phi$ falsifies the identity in $\Gamma(\mathbb{T})$.

Conversely, suppose that $f_{u, \rho}^w$ and $f_{u, \rho}^v$ are equivalent for all $u, \rho$ with $u\in \Sigma^*$, and $\rho$ a path of length $|u|$ through $\Gamma$. By Proposition~\ref{prop_finitarywilldoGamma} it suffices to show that the identity $w=v$ is satisfied by every morphism $\phi : \Sigma^+ \to \Gamma(\ft)$,
so let $\phi$ be such a morphism. Define $\underline{x} \in \mathbb{FT}^{\Sigma \times \Gamma}$ by $\underline{x}(s,i) = \phi(s)_{i,i}$.

Since $\phi$ is a morphism to $\Gamma(\ft)$, we know that $\phi(w)_{i,j}=-\infty = \phi(v)_{i,j}$ whenever $i \not\leq j$.

On the other hand, if $i \leq j$ then Lemma~\ref{lemma_entriesGamma} gives
$$\phi(w)_{i,j} \ = \ \bigoplus_{\substack{u \in \Sigma^*,\\ |u| \leq n-1}} \bigoplus_{\rho \in \Gamma_{i,j}^{|u|}} \left(\bigotimes_{k=1}^{|u|} \phi(u_k)_{\rho_{k-1}, \rho_k}\right) \otimes f_{u, \rho}^w(\underline{x}) \ = \ \phi(v)_{i,j}.$$
\end{proof}

\begin{theorem} \label{thm_UTGamma}
Let $\Gamma$ be a partial order with maximum chain length $n$. Then $\Gamma(\trop)$ satisfies exactly the same semigroup identities as $UT_n(\trop)$.
\end{theorem}
\begin{proof}
Let $\rho=(\rho_1, \ldots, \rho_n)$ be a maximal length path in $\Gamma$, and let $\tau_i = (\rho_1, \ldots, \rho_i)$ for $i=1, \ldots, n$. We note that if $\rho'$ and $\rho''$ are paths of the \emph{same length}, then $f_{u, \rho''}^w$ can be obtained from $f_{u, \rho'}^w$ by the change of variables $x(s, \rho'_i) \mapsto x(s, \rho''_i)$. Thus $f_{u,\rho'}^w = f_{u,\rho'}^v$ if and only if $f_{u,\rho''}^w = f_{u,\rho''}^v$.  It follows from this observation together with Theorem~\ref{thm_kletterGamma} that $w=v$ in $\Gamma(\mathbb{T})$ if and only if $f_{u,\tau_i}^w = f_{u,\tau_i}^v$ for all words $u \in \Sigma^*$ of length $i$ for all $i=0, \ldots, n-1$. In particular, we note that $w=v$ in $\Gamma(\mathbb{T})$ if and only if $w=v$ in $UT_{n}(\mathbb{T})$.\end{proof}

We remark that Theorem~\ref{thm_UTGamma} can also be deduced as a corollary of Theorem~\ref{thm_realisation} below (the proof
of which is independent) together with Birkhoff's HSP Theorem \cite{B35}.

\section{Chain-Structured Semigroups as Divisors of $UT_n(\mathbb{T})$}\label{sec_divisors}

As in the previous section, let $\Gamma$ be a partial order with finite maximum chain length $n$.
Theorem~\ref{thm_UTGamma} says that the chain-structured tropical matrix semigroup $\Gamma(\mathbb{T})$ satisfies exactly
the same identities as the upper triangular tropical matrix semigroup $UT_n(\mathbb{T})$.
By Birkhoff's HSP Theorem \cite{B35}, this means that $\Gamma(\mathbb{T})$ must be realisable as a homomorphic image of a subsemigroup of
a direct power of $UT_n(\mathbb{T})$. In this section we show how to construct an explicit realisation of $\Gamma(\mathbb{T})$ in
this way.

For each $k=1, \ldots, n$ and each $k$-vertex path $\rho$ in $\Gamma$, let $\Delta_\rho$ be an isomorphic copy of the semigroup $UT_k(\mathbb{T})$ but with rows and columns indexed by
the vertices occuring in $\rho$, and the order of vertices in $\rho$ used to determine the upper triangular structure in the obvious way (so that
for any $M \in \Delta_\rho$ we have $M_{i,j} = -\infty$ if $j$ comes strictly before $i$ in $\rho$).

Let $\Delta$ be the direct product of the semigroups $\Delta_\rho$ as $\rho$ varies over all directed paths in $\Gamma$. Given $D \in \Delta$ we write $D_\rho$ for projection of $D$ onto the $\rho$-coordinate (so $D_\rho \in \Delta_\rho$). For $i,j \in \rho$
we write $D_{\rho,i,j}$ for the $(i,j)$ entry of $D_{\rho}$.

Notice that since each $\Delta_\rho$ is isomorphic to $UT_k(\mathbb{T})$ for some $k \leq n$, $\Delta$ embeds naturally in a direct
power of $UT_n(\mathbb{T})$.

We define a function (which we do not claim to be a morphism)
$$\psi : \Gamma(\mathbb{T}) \to \Delta,  \ \ \ \psi(M)_{\rho,i,j} = M_{ij}.$$
This function is not surjective; indeed its image need not even be a subsemigroup of $\Delta$. Let $\Psi$ be the subsemigroup of $\Delta$ generated by the image of $\psi$ and define
$$\varphi : \Psi \to \Gamma(\mathbb{T}), \ \ \ \ \varphi(D)_{i,j} = {\rm sup}\{D_{\rho,i,j}: i,j \in \rho\}.$$
It is immediate from the definitions that $\psi$ is an injective function, and that $\varphi(\psi(M)) = M$ for all $M \in \Gamma(\mathbb{T})$. Moreover, the boundedness conditions on $\Gamma(\mathbb{T})$ ensure that $\varphi$ is well-defined on $\Psi$. We show that $\varphi$ is a surjective semigroup morphism.

\begin{theorem}\label{thm_realisation}
With notation as above, $\varphi$ is a surjective semigroup morphism from $\Psi$ (which is a subsemigroup
of $\Delta$ and hence embeds naturally in a direct power of $UT_n(\mathbb{T})$) onto $\Gamma(\mathbb{T})$.
\end{theorem}
\begin{proof}
Since every element of $\Psi$ has the
form $\psi(M_1) \cdots \psi(M_k)$ for some $M_1, \dots, M_k \in \Gamma(\mathbb{T})$ and we know that
$\varphi(\psi(M)) = M$ for all $M \in \Gamma(\mathbb{T})$, it suffices to take $M_1, \dots, M_k \in \Gamma(\mathbb{T})$ and show that
$$\varphi(\psi(M_1) \cdots \psi(M_k)) \ = \ M_1 \otimes \cdots \otimes M_k.$$

Consider the $(i,j)$ entry of the left-hand side. By definition, this is the supremum over all paths $\rho$ containing $i$ and $j$ of the $(i,j)$ entry of $(\psi(M_1) \dots \psi(M_k))_\rho \in \Delta_\rho$.
Since $\Delta$ is a direct product, we have
$$(\psi(M_1) \cdots \psi(M_k))_\rho \ = \ (\psi(M_1)_\rho)\otimes  \cdots \otimes (\psi(M_k))_\rho$$
in $\Delta_\rho$. By definition of matrix multiplication, the $(i,j)$ entry of this is the supremum over all
$(k+1)$-vertex walks $(i = w_0, w_1, \dots, w_{k-1}, w_k = j)$ within the vertex set of $\rho$ of
$$(M_1)_{w_0 w_1}\otimes \cdots \otimes(M_k)_{w_{k-1} w_k}.$$
But since every $(k+1)$-vertex walk is contained in the vertex set of \textit{some} path $\rho$, this means that the $(i,j)$ entry of the left-hand-side is the
supremum over all $(k+1)$-vertex walks $(i = w_0, w_1, \dots, w_{k-1}, w_k = j)$ in $\Gamma$ of
$$(M_1)_{w_0 w_1}\otimes  \cdots \otimes (M_k)_{w_{k-1} w_k}.$$
But by the definition of multiplication in $\Gamma(\mathbb{T})$, this is exactly the $(i,j)$ entry of the right-hand side.
\end{proof}
As mentioned above, the proof of Theorem~\ref{thm_realisation} is entirely independent of Theorem~\ref{thm_UTGamma}, and in fact the
latter can (as an alternative proof strategy) be deduced from the former.

\section{The Free Monogenic Inverse Monoid}\label{sec_monogenicinverse}

In this section we consider tropical representations of the free one-generated object in the category of inverse monoids. This monoid, which
we denote $\mathcal{I}$, admits several representations, but most conveniently for our purposes it is isomorphic (see
for example \cite[Chapter 5, Exercise 42]{H95}) to the set of
triples of integers:
$$\lbrace (i,j,k) \in \mathbb{Z}^3 \ \mid \ i, j \geq 0, \;\; -j \leq k \leq i \rbrace$$
with multiplication given by:
$$(i,j,k)(i',j',k') \ = \ ({\rm max}(i,i'+k), \ {\rm max}(j,j'-k), \ k+k').$$

Consider now the set $\Gamma = \lbrace 1,2,3 \rbrace$ equipped with the partial order $\preceq$ in which
$1, 2 \preceq 3$ but $1$ and $2$ are incomparable. It is easy to see, from the above representation, that there
is an embedding of semigroups
$$\mathcal{I} \hookrightarrow \Gamma(\ft) \subseteq UT_3(\trop), \ \ \  (i,j,k) \mapsto \left(\begin{array}{ c c c}
k & -\infty& i \\
-\infty& -k& j \\
-\infty& -\infty& 0
\end{array}\right).$$

Since $\Gamma$ has maximum chain length $2$, we can deduce by Theorem~\ref{thm_UTGamma} that $\mathcal{I}$ satisfies
all the identities satisfied by $UT_2(\trop)$, and hence by Theorem \ref{main_theorem}, all identities satisfied in $\mathcal{B}$. (In fact this was known as a consequence of work of Scheiblich \cite{Sch71} which showed that $\mathcal{I}$ can be embedded in a direct product of copies of $\mathcal{B}$.) The converse also holds: since $\mathcal{B}$ is itself a monogenic inverse monoid, it is a homomorphic image of $\mathcal{I}$, and hence satisfies all identities satisfied in the latter.

The above representation of $\mathcal{I}$ in $\Gamma(\ft)$ is of course also a representation in $UT_3(\trop)$. In view of the preceding remarks it is natural to ask if one can go one step better and find a faithful representation of $\mathcal{I}$ in $UT_2(\ft)$. It transpires that this cannot be done.

\begin{proposition}
The free monogenic inverse monoid $\mathcal{I}$ embeds in $UT_n(\trop)$ if and only if $n \geq 3$.
\end{proposition}

\begin{proof}
We have seen that $\mathcal{I}$ embeds into $UT_3(\trop)$ which in turn embeds in $UT_n(\trop)$ for all $n \geq 3$.  To see that $\mathcal{I}$ does not embed into $UT_2(\mathbb{T})$ we consider the structure formed by the idempotents of each semigroup under the natural partial order given by
$e \leq f $ if and only if $e=ef=fe$.

It
can readily be verified that every idempotent in $UT_2(\trop)$ is of one the forms
\[ \scalebox{0.85}{$Z=\left(\begin{array}{c c}
-\infty & -\infty\\
-\infty& -\infty
\end{array}\right), \ G_x=\left(\begin{array}{c c}
0 & x\\
-\infty& 0
\end{array}\right), \ E_x=\left(\begin{array}{c c}
0 & x\\
-\infty& -\infty
\end{array}\right), \ F_x=\left(\begin{array}{c c}
-\infty & x\\
-\infty& 0
\end{array}\right),$
}\]
where $x \in \mathbb{T}$, and that the natural partial order on idempotents is given by
\begin{itemize}
\item[(i)] $Z \leq X \leq G_{-\infty}$ for all idempotent elements $X$;
\item[(ii)] $E_x, F_x, G_x \leq G_y$ if and only if $y \leq x$; and
\item[(iii)] idempotents of the form $E_x$ and $F_x$ are incomparable with all remaining idempotents.
\end{itemize}
Notice that the only idempotents which lie strictly above two or more distinct idempotents are those of the form $G_x$, and that these idempotents are totally ordered. Thus, it is not possible within $UT_2(\trop)$ to find two incomparable idempotents and two other idempotents which lie strictly below both of them. In contrast,  $(2,3,0)$ and $(3,2,0)$ are incomparable idempotents of $\mathcal{I}$ which both lie above the idempotents $(3,3,0)$ and $(4,4,0)$. Since any embedding of semigroups strictly preserves the natural partial order, this means $\mathcal{I}$ cannot embed in $UT_2(\trop)$.
\end{proof}
(The above proof replaces an erroneous proof in the the published version of this article. The latter erred in claiming that the idempotents of the form $E_x$ and $F_x$ are incomparable with all other idempotents except $G_{-\infty}$ and $Z$. That this is false can be seen from the above proof. We are grateful to Duarte Ribeiro (private communication) for alerting us to the error.)

\section{Complexity of Checking Identities}\label{sec_complexity}

Deciding whether a given identity holds in a (fixed) semigroup is often computationally hard. It follows from a result of Murskii \cite{Murskii} that there is a semigroup for which this problem is undecidable. Even in small finite semigroups the problem may be intractable: indeed, it is coNP-complete for both the full transformation semigroup $T_3$ \cite{Almeida} and the symmetric group $S_5$ \cite{Horvath}.

In this section we show how our results can be used to derive efficient algorithms for checking whether a given identity holds in $UT_n(\mathbb{T})$, and hence (by Theorem~\ref{thm_UTGamma}) in chain-structured matrix semigroups in general. To be precise, the results above (Theorem~\ref{thm_kletter}, Corollary~\ref{cor_kletter} and Theorem~\ref{thm_kletterGamma}) allow us to reduce the problem to checking whether formal tropical polynomials define
the same function, and in this section we show how the latter problem can be reduced to (real) linear programming. For a fixed $n$
the resulting algorithms run in time polynomial in the size of the alphabet and the length of the identity, although the degree of the
polynomial rises (and hence the complexity increases exponentially) as $n$ grows.

One special case is of particular importance: our Main Theorem implies that the algorithm for $UT_2(\mathbb{T})$ can also be used to check (in polynomial time) whether a given identity holds in the bicyclic monoid (and hence also the free monogenic inverse monoid). Another algorithm for this problem was given by Pastijn \cite{P06};  in fact
although this algorithm arose from a different approach to an ostensibly different problem, it is very closely related to our own. No complexity analysis is given in \cite{P06}, but Pastijn's algorithm also essentially reduces to linear programming, and hence should
also be implementable in polynomial time.

We begin by establishing some basic properties of many-variable tropical polynomials. The ideas concerned are well-known to
max-plus algebraists in the one-variable case (see \cite[Section 5.1]{Butkovic}) and are increasingly studied by algebraic geometers in the many-variable case. However, we have not been able to locate the precise statements we require in the literature, so we present (without any particular claim of originality) a concise self-contained exposition.

A term in a formal tropical polynomial is called \textit{essential} if there is some value of the variables for which it is the only term to attain the maximum, and \textit{inessential} otherwise. For example, as discussed in Section~\ref{sec_defs}, $x$ is an inessential term in $x^{2} \oplus x \oplus 1$. The formal polynomial itself is called \textit{essential} if all of its terms are essential. We shall need the following observation, which follows from \cite[Proposition 1.5 and Lemma 1.6]{IM10}:

\begin{lemma}\label{lemma_essential}
Every formal tropical polynomial is equivalent to a unique essential formal tropical polynomial.
\end{lemma}

\begin{lemma}\label{lemma_findessential}
There is an algorithm which, given a formal tropical polynomial in $k$ variables with $m$ terms with one distinguished term, decides whether the distinguished term is essential in time polynomial in $k$ and $m$.
\end{lemma}
\begin{proof}
It suffices to check if there are real values of the variables which make the distinguished term simultaneously exceed each of the other $m-1$ terms. Since each term of a tropical polynomial is a classical linear function, this is just a (classical) linear programming problem of checking the
solvability of $m-1$ linear inequalities in $k$ real variables, which is solvable in polynomial time (see for example \cite{Aspvall}).
\end{proof}

\begin{theorem}
Let $n \in \mathbb{N}$ be a fixed positive integer.
Then there is an algorithm which, given an identity $v=w$ over a finite alphabet $\Sigma$, decides in time polynomial in $|v|+|w|$ and $|\Sigma|$ whether the identity holds in $UT_n(\mathbb{T})$.
\end{theorem}
\begin{proof}
Suppose we are given an identity $v=w$ where $v$ and $w$ are words over a $k$-letter alphabet $\Sigma$ and $|v| + |w| = m$. By
Theorem~\ref{thm_kletterGamma} it suffices to compute the formal polynomials $f_{u,\rho}^v$ and $f_{u,\rho}^w$ for $u \in \Sigma^*$ with
$|u| < n$ and $\rho$ a ($|u|$+1)-vertex path in the partially ordered set $\lbrace 1, \dots, n \rbrace$, and check whether each $f_{u,\rho}^v$ is equivalent to the corresponding $f_{u,\rho}^w$. The number of pairs of polynomials to consider is bounded above by the number of ways to choose
$u$ (which since $n$ is fixed is a polynomial function of $k=|\Sigma|$) times the number of paths of length $n$ or less in $\lbrace 1, \dots, n \rbrace$ (which since
$n$ is fixed is a constant). Moreover, it is easy to see from the definition that each tropical polynomial $f_s^w$ can be computed in time polynomial in
$|\Sigma|$ and $|w|$, and so in particular has polynomially many terms.

Now by Lemma~\ref{lemma_findessential} we may check in polynomial time which terms in our functions are essential; by discarding
the inessential terms we obtain, again in polynomial time, essential formal polynomials representing the same functions. By
Lemma~\ref{lemma_essential} it suffices to check if these are the same as formal polynomials, which again can clearly be done
in polynomial time.
\end{proof}

We emphasise that the algorithm runs in polynomial time only for a fixed semigroup $UT_n(\mathbb{T})$; allowing $n$ to grow results in exponential growth in the number of the pairs of formal polynomials which must be compared, so time complexity as a function of $n$ is exponential.

In the case $n=2$ we may slightly reduce the number of polynomials to be considered by using Theorem~\ref{thm_kletter}
in place of Theorem~\ref{thm_kletterGamma}. In the special case that $n = 2$ and $|\Sigma| = 2$ (hence for two-letter identities
over the bicyclic monoid) Corollary~\ref{cor_kletter} allows us to reduce the problem to checking equivalence of four pairs of \textit{single-variable} polynomial equations, each of size linear in $|v|+|w|$. Equivalence of single-variable tropical polynomials can be checked
even more efficiently: in linear time, presuming a random access model of computation with unit time arithmetic and comparison of numbers \cite[Algorithm 5.1.11]{Butkovic}. Thus, $2$-letter identities in $UT_2(\mathbb{T})$ (and hence, by Theorem \ref{main_theorem}, in $\mathcal{B}$) can be checked in linear time given in such a model of computation.

\end{document}